\documentclass[a4paper,10pt]{amsart}
\usepackage{amssymb,amscd,amsmath}

\theoremstyle{plain}
\newtheorem{thm}{Theorem}[section]

\newtheorem{lem}[thm]{Lemma}

\theoremstyle{definition}

\theoremstyle{remark}
\newtheorem{rem}{Remark}[section]
\newtheorem{ex}[rem]{Example}

\newcommand{\forme}[1]{}

\title{Terwilliger algebras of wreath products by quasi-thin schemes}

\author[K.~Kim]{Kijung Kim}
\address{Department of Mathematics, Pusan National University, Busan 609-735, Republic of Korea}
\email{knukkj@pusan.ac.kr}

\date{\today}

\begin{document}
\maketitle

\begin{abstract}
The structure of Terwilliger algebras of wreath products by thin schemes or one-class schemes was studied in [A. Hanaki, K. Kim, Y. Maekawa, Terwilliger algebras of direct and wreath products of association schemes, J. Algebra 343 (2011) 195--200].
In this paper, we will consider the structure of Terwilliger algebras of wreath products by quasi-thin schemes.
This gives a generalization of their result.
\end{abstract}

\footnotetext{This work was supported by the National Research Foundation of Korea Grant funded by the Korean Government(Ministry of Education, Science and Technology) [NRF-2010-355-C00002].}

\section{Introduction}\label{sec:intro}
The Terwilliger algebra is a new algebraic tool for the study of
association schemes introduced by P. Terwilliger in \cite{terwilliger, terwilliger2, terwilliger3}.
In general, this algebra is a non-commutative, finite dimensional, and semisimple $\mathbb{C}$-algebra.
In the theory of association schemes, the wreath product is a method to construct new association schemes.
Recently G. Bhattacharyya, S.Y. Song and R. Tanaka began to study Terwilliger algebras of wreath products of one-class association schemes in \cite{song}.
In particular, S.Y. Song and B. Xu gave a complete structural description of Terwilliger algebras for wreath products of one-class association schemes in \cite{sx}.
Terwilliger algebras of wreath products by thin schemes or one-class schemes were studied in \cite{hkm}.
In this paper, we give a generalization of their result by replacing thin schemes with quasi-thin schemes.

The remainder of this paper is organized as follows.
In Section~\ref{sec:pre}, we review notations and basic results on coherent configurations and Terwilliger algebras as well as important results on quasi-thin schemes.
In Section~\ref{sec:main}, based on the fact that one point extensions of quasi-thin schemes coincide with their Terwilliger algebras, we determine all central primitive idempotents of Terwilliger algebras of wreath products by quasi-thin schemes.
In Section~\ref{sec:maintheorem}, we state our main theorem.

%%%%%%%%%%%%%%%%%%%%%%%%%%%%%%%%%%%%%%%%%%%%%%%%%%%%%%%%%%%%%%%%%%%%%%%%%%%%%%%%
\section{Preliminaries}\label{sec:pre}
In this section, to unify notations and terminologies given in \cite{ep, hkm, kmpwz, mp}, we combine them.
We assume that the reader is familiar with the basic notions of association schemes in \cite{zies}.

\subsection{Coherent configurations and coherent algebras}
Let $X$ be a finite set and $S$ a partition of $X \times X$.
Put by $S^\cup$ the set of all unions of the elements of $S$.
A pair $\mathcal{C}=(X,S)$ is called a \textit{coherent configuration} on $X$. if the following conditions hold:
\begin{enumerate}
\item[(1)] $1_X := \{(x,x)\mid x\in X \} \in S^\cup$.
\item[(2)] For $s \in S$, $s^{\ast} :=\{(y,x)\mid (x,y)\in s\}\in S$.
\item[(3)] For all $s,t,u\in S$ and all $x, y \in X$,
 $$p_{st}^{u} :=| \{z\in X\mid (x,z)\in s, \ (z,y)\in t\} |$$
is constant whenever $(x,y)\in u$.
\end{enumerate}
The elements of $X$, $S$ and $S^\cup$ are called the \textit{points}, the \textit{basis relations} and the \textit{relations}, respectively.
The numbers $|X|$ and $|S|$ are called the \textit{degree} and \textit{rank}.
Any set $\Delta \subseteq X$ for which $1_\Delta \in S$ is called the \textit{fiber}. The set of all fibers is denoted by $\mathrm{Fib}(\mathcal{C})$.
The coherent configuration $\mathcal{C}$ is called \textit{homogeneous} or a \textit{scheme} if $1_X \in S$.
If $Y$ is a union of fibers, then the \textit{restriction} of $\mathcal{C}$ to $Y$ is defined to be a coherent configuration
$$\mathcal{C}_Y =(Y,S_Y),$$
where $S_Y$ is the set of all non-empty relations $s \cap (Y\times Y)$ with $s \in S$.
For $s \in S$, let $\sigma_{s}$ denote the matrix in $\mathrm{Mat}_{X}(\mathbb{C})$ that has entries
\[(\sigma_{s})_{xy}  = \left\{
                      \begin{array}{ll}
                      1 & \hbox{if $(x,y) \in s$;} \\
                      0 & \hbox{otherwise.}
                      \end{array}
                     \right.\]
We call $\sigma_{s}$ the \textit{adjacency matrix} of $s \in S$.
Then $\bigoplus_{s\in S}\mathbb{C}\sigma_{s}$ becomes a subalgebra of
$\mathrm{Mat}_X(\mathbb{C})$. We call $\bigoplus_{s\in S}\mathbb{C}\sigma_{s}$ the
\textit{adjacency algebra} of $S$, and denote it by $\mathcal{A}(S)$.
A linear subspace $\mathcal{A}$ of $\mathrm{Mat}_X (\mathbb{C})$ is called a \textit{coherent algebra} if the following conditions hold:
\begin{enumerate}
\item[(1)] $\mathcal{A}$ contains the identity matrix $I_X$ and the all-one matrix $J_X$.
\item[(2)] $\mathcal{A}$ is closed with respect to the ordinary and Hadamard multiplications.
\item[(3)] $\mathcal{A}$ is closed with respect to transposition.
\end{enumerate}
Let $B$ be the set of primitive idempotents of $\mathcal{A}$ with respect to the Hadamard multiplication.
Then $B$ is a linear basis of $\mathcal{A}$ consisting of $\{0, 1 \}$-matrices such that
$$\sum_{s \in B} \sigma_{s} = J_X  ~\text{and}~  \sigma_{s} \in B \Leftrightarrow \sigma_{s}^t \in B.$$

\begin{rem} \label{rem:1}
There are bijections between the sets of coherent configurations and coherent algebras as follows:
$$S \mapsto \mathcal{A}(S) ~\text{and}~ \mathcal{A}\mapsto \mathcal{C}(\mathcal{A}),$$
where $\mathcal{C}(\mathcal{A}) = (X, S')$ with $S'=\{ s \in X\times X \mid \sigma_{s} \in B\}$.
\end{rem}

Let $\mathcal{C}=(X,S)$ be a coherent configuration. For each $x \in X$, we define $xs:= \{y \in X \mid (x,y) \in s \}$. A point $x \in X$ is called \textit{regular} if
$$|xs| \leq 1, s \in S.$$
In particular, if the set of all regular points is non-empty, then $\mathcal{C}$ is called \textit{1-regular}.

Let $\mathcal{C}=(X,S)$ be a scheme and $T$ a closed subset containing the thin residue of $S$.
Put by $S_{(T)}$ the set of all basic relations $s_{\Delta, \Gamma} := s \cap (\Delta \times \Gamma)$, where $s \in S$ and $\Delta, \Gamma \in X/T:=\{ xT \mid x \in X\}$.
Then by \cite{ep2} or \cite{kmpwz}, the pair $\mathcal{C}_{(T)}=(X,S_{(T)})$ is a coherent configuration called the \textit{thin residue extension}
of $\mathcal{C}$.

\subsection{Terwilliger algebras and one point extensions}
Let $(X,S)$ be a scheme. For $U\subseteq X$, we denote by $\varepsilon_{U}$ the diagonal matrix in $\mathrm{Mat}_{X}(\mathbb{C})$ with
entries $(\varepsilon_{U})_{xx}=1$ if $x\in U$ and $(\varepsilon_{U})_{xx}=0$ otherwise. Note that $J_{U,V}:= \varepsilon_{U} J_X \varepsilon_{V}$ and $J_U := J_{U,U}$ for $U, V \subseteq X$.

The  \textit{Terwilliger algebra} of $(X,S)$ with respect to $x_{0}\in X$ is defined as a subalgebra of $\mathrm{Mat}_X(\mathbb{C})$ generated by $\{\sigma_{s} \mid s\in S\}\cup \{\varepsilon_{x_{0}s} \mid s \in S\}$ (see \cite{terwilliger}).
The Terwilliger algebra will be denoted by $\mathcal{T}(X,S,x_{0})$ or $\mathcal{T}(S)$.
Since $\mathcal{A}(S)$ and $\mathcal{T}(S)$ are closed under transposed conjugate,
they are semisimple $\mathbb{C}$-algebras.
The set of irreducible characters of $\mathcal{T}(S)$ and $\mathcal{A}(S)$ will be denoted by
$\mathrm{Irr}(\mathcal{T}(S))$ and $\mathrm{Irr}(\mathcal{A}(S))$,
respectively. The \textit{trivial character} $1_{\mathcal{A}(S)}$ of $\mathcal{A}(S)$
is a map $\sigma_{s}\mapsto n_{s}$, where $n_{s}:=p_{ss*}^{1_X}$ is called the \textit{valency} of $s$, and the corresponding central
primitive idempotent is $|X|^{-1}J_{X}$. The \textit{trivial character} $1_{\mathcal{T}(S)}$ of $\mathcal{T}(S)$
corresponds to the central primitive idempotent $\sum_{s\in S}n_{s}^{-1}\varepsilon_{x_{0}s}J_{X}\varepsilon_{x_{0}s}$ of
$\mathcal{T}(S)$. For $\chi\in \mathrm{Irr}(\mathcal{A}(S))$ or $\mathrm{Irr}(\mathcal{T}(S))$, $e_{\chi}$ will be the corresponding
central primitive idempotent of $\mathcal{A}(S)$ or $\mathcal{T}(S)$.
For convenience, we denote $\mathrm{Irr}(\mathcal{A}(S)) \setminus \{1_{\mathcal{A}(S)}\}$ and $\mathrm{Irr}(\mathcal{T}(S)) \setminus \{1_{\mathcal{T}(S)}\}$ by $\mathrm{Irr}(\mathcal{A}(S))^\times$ and $\mathrm{Irr}(\mathcal{T}(S))^\times$, respectively.

Let $\mathcal{C}=(X,S)$ be a coherent configuration and $x \in X$. Denote by $S_x$ the set of basic relations of the smallest coherent configuration on $X$ such that
$$1_x \in S_x ~\text{and}~ S \subset S_x^\cup.$$
Then the coherent configuration $\mathcal{C}_x=(X,S_x)$ is called a \textit{one point extension} of $\mathcal{C}$.
It is easy to see that given $s, t, u \in S$ the set $xs$ and the relation $u_{xs,st}$ are unions of some fibers and some basic relations of $\mathcal{C}_x$, respectively.

\begin{rem} \label{rem:2}
A one point extension $\mathcal{C}_x$ of a scheme $\mathcal{C}$ is related to $\mathcal{T}(X,S,x)$. In fact, $\mathcal{C}_x \supseteq \mathcal{T}(X,S,x)$.
\end{rem}

\subsection{Direct sums, direct products and wreath products}
Let $\mathcal{C}=(X,S)$ and $\mathcal{C'}=(X',S')$ be coherent configurations.
Put by $X\sqcup X'$ the disjoint union of $X$ and $X'$, and by $S \boxplus S'$ the union of the set $S\cup S'$ and the set of all relations
$\Delta \times \Delta', \Delta' \times \Delta$, where $\Delta, \Delta'$ are fibers of
$\mathcal{C}$ and $\mathcal{C'}$, respectively.
Then the pair
$$\mathcal{C} \boxplus \mathcal{C'} = (X\sqcup X', S\boxplus S')$$
is a coherent configuration called the \textit{direct sum} of $\mathcal{C}$ and $\mathcal{C'}$.
Set $S \times S' = \{ s\times s' \mid s \in S, s' \in S' \}$, where $s\times s'$ is the relation on $X\times X'$ consisting of all pairs
$((\alpha, \alpha'), (\beta, \beta'))$ with $(\alpha, \beta) \in s$ and $(\alpha', \beta') \in s'$.
Then the pair
$$\mathcal{C} \times \mathcal{C'} = (X \times X', S \times S')$$
is a coherent configuration called the \textit{direct product} of $\mathcal{C}$ and $\mathcal{C'}$.
The adjacency matrix of $s \times s' \in S\times S'$ is given by the Kronecker product $\sigma_{s}\otimes \sigma_{s'}$.

Let $(X,S)$ and $(Y,T)$ be schemes.
For $s\in S$, set
$\tilde{s}=\{((x,y),(x^{\prime},y)) \mid (x,x^{\prime})\in s, \ y\in
Y\}$.
For $t\in T$, set
$\bar{t}=\{((x,y),(x^{\prime},y^{\prime})) \mid x,x^{\prime}\in X, \
(y,y^{\prime})\in t\}$. Also set $S\wr T=\{\tilde{s}\mid s\in
S\}\cup\{\bar{t}\mid t\in T\setminus\{1_Y\}\}$.
Then $(X\times Y, S\wr T)$ is a scheme called the
\textit{wreath product} of $(X,S)$ by $(Y,T)$.
For the adjacency matrices, we have $\sigma_{\tilde{s}}=\sigma_{s}\otimes I_{Y}, \sigma_{\bar{t}}=J_{X}\otimes \sigma_{t}$.
Note that
$(x_{0},y_{0})\tilde{s} =(x_{0}s,y_{0}) =\{(x,y_{0}) \mid x\in x_{0}s\}$,
$(x_{0},y_{0})\bar{t} =(X,y_{0}t)  =\{(x,y) \mid x\in X, \ y\in y_{0}t\}$
and $\varepsilon_{(x_{0},y_{0})\tilde{s}} =\varepsilon_{x_{0}s}\otimes \varepsilon_{y_{0}1_Y}$,
$\varepsilon_{(x_{0},y_{0})\bar{t}} =\sum_{s\in S}\varepsilon_{x_{0}s}\otimes \varepsilon_{y_{0}t} =I_{X}\otimes \varepsilon_{y_{0}t}$.

\subsection{Quasi-thin schemes}
A scheme $\mathcal{C}=(Y,T)$ is called \textit{quasi-thin} if $T = T_1 \cup T_2$, where $T_i$ is the set of basic relations with valency $i$ ($i \in \{1, 2 \}$).

\begin{lem}\label{lem:basic1}(\cite{hirasaka}, Lemma 4.1)
For $t \in T_2$, there exists a unique $t^{\bot}$ such that $tt^* = \{1_Y, t^{\bot} \}$.
\end{lem}
In \cite{mp}, any element from the set $T^\bot = \{ t^\bot \mid t \in T_2 \}$ is called an \textit{orthogonal} of $\mathcal{C}$.
If $|T^{\bot}| = 1$ and $T^{\bot} \subseteq T_1$, then $H := \{1_{Y} \} \cup T^\bot$ is the thin residue.
Considering $\mathcal{C}_{(H)}=(Y,T_{(H)})$, it follows that given $\Delta, \Gamma \in \mathrm{Fib}(\mathcal{C}_{(H)})$ either the set $\Delta \times \Gamma
\in T_{(H)}$ or the set $\Delta \times \Gamma \not \in T_{(H)}$.
Denote the latter case by $\Delta \sim \Gamma$. Then $\sim$ is an equivalence relation on the set $\mathrm{Fib}(\mathcal{C}_{(H)})$.

Next, we state two results on quasi-thin schemes.

\begin{thm}\label{thm:coherent-terwilliger}(\cite{mp}, Theorem 6.1)
Let $\mathcal{C}=(Y,T)$ be a quasi-thin scheme and $y_{0} \in Y$.
Then $\mathcal{A}(T_{y_{0}}) = \mathcal{T}(Y,T,y_{0})$.
\end{thm}

\begin{thm}\label{thm:3type}(\cite{mp}, Theorem 5.2 and Corollary 6.4)
Let $\mathcal{C}=(Y,T)$ be a quasi-thin scheme with $T^{\bot} \neq \emptyset$.
\begin{enumerate}
\item[(1)] If $|T^{\bot}| = 1$ and $T^{\bot} \subseteq T_2$, then
$$T = T_1 \{1_Y, u \},$$
where $y_{0} \in Y$ and $T^{\bot} = \{ u \}$.

\item[(2)] If $|T^{\bot}| = 1$ and $T^{\bot} \subseteq T_1$, then
$$\mathcal{C}_{(H)} = \boxplus_{i \in I} \mathcal{C}_i,$$
where $I$ is the set of classes given in the above $\sim$, $Y_i$ is the union of fibers belonging a class $i \in I$, and $\mathcal{C}_i = (\mathcal{C}_{(H)})_{Y_i}$.

\item[(3)] If $|T^{\bot}| \geq 2$, then $\mathcal{C}_{y_0}=(Y, T_{y_0})$ is 1-regular. In particular, any point of $y_0 T_2$ is regular.
\end{enumerate}
\end{thm}

\begin{ex}\label{qexam}
Some examples for each case of Theorem \ref{thm:3type} can be found in \cite{hanakimi}.
\begin{enumerate}
\item[(1)] as12 No.51
\item[(2)] as12 No.48
\item[(3)] as28 No.175, No 176
\end{enumerate}
\end{ex}

%\begin{thm}\label{thm:3type-regular}
%Let $\mathcal{C}=(Y,T)$ be a quasi-thin scheme with $T^{\bot} \neq \emptyset$.
%\end{thm}

%The above statement will be used in the proof of ........

%%%%%%%%%%%%%%%%%%%%%%%%%%%%%%%%%%%%%%%%%%%%%%%%%%%%%%%%%%%%%%%%%%%%%%%%%%%%%%%%
\section{Wreath products by quasi-thin schemes}\label{sec:main}

Let $(X,S)$ and $(Y,T)$ be schemes. Fix $x_{0}\in X$ and
$y_{0}\in Y$ and consider $(X\times Y, S\wr T)$ and $\mathcal{T}(X\times Y, S \wr T, (x_0, y_0))$.
In the rest of this section, we assume that $(Y,T)$ is a quasi-thin scheme with $T^{\bot} \neq \emptyset$.

%%%%%%%%%%%%%%%%%%%%%%%%%%%%%%%%%%%%%%%%%%%%%%%%%%%%%%%%%%%%%%%%%%%%%%%%%%%%%%%%%
\subsection{The restriction of $\mathcal{T}(S \wr T)$ to $X \times (Y\setminus \{y_0\})$}\label{sec:res}

$\mathcal{T}(S \wr T)$ is generated by $\{ J_{X}\otimes \sigma_{t}, I_{X}\otimes \varepsilon_{y_{0}t} \mid t \in T\setminus \{1\} \}
\cup \{ \sigma_{s}\otimes I_{Y}, \varepsilon_{x_{0}s}\otimes \varepsilon_{\{y_{0}\}} \mid s \in S \}$.

Since $\sum_{s \in S} \varepsilon_{x_{0}s}\otimes \varepsilon_{\{y_{0}\}} = I_{X}\otimes \varepsilon_{\{y_{0}\}}$ and
$\sum_{s \in S} \sigma_{s}\otimes I_{Y} = J_{X}\otimes I_{Y}$, we consider a subalgebra $\mathcal{U}$ generated by $\{ J_{X}\otimes \sigma_{t}, I_{X}\otimes \varepsilon_{y_{0}t} \mid t \in T \}$. It is easy to see that $\mathcal{U}$ is generated by $\{ |X|^{-1}J_{X}\otimes \varepsilon_{y_{0}t_1} \sigma_{t} \varepsilon_{y_{0}t_2} \mid t_1, t_2 \in T \}$ and isomorphic to $\mathcal{T}(T)$.
So by Theorem \ref{thm:coherent-terwilliger}, a basis $B(\mathcal{U})$ of $\mathcal{U}$ can be determined by $\mathcal{C}(\mathcal{T}(T))$, i.e. $B(\mathcal{U}) = \{ J_{X}\otimes \sigma_{c} \mid c \in \mathcal{R}\}$, where $\mathcal{R}$ is the set of basic relations of $\mathcal{C}(\mathcal{T}(T))$.

We consider $\varepsilon_{X}\otimes \varepsilon_{Y\setminus  \{y_{0}\}} \mathcal{T}(S \wr T) \varepsilon_{X}\otimes \varepsilon_{Y\setminus  \{y_{0}\}}$.
Since $\mathcal{T}(S \wr T)$ is generated by $B(\mathcal{U}) \cup \{ \sigma_{s}\otimes I_{Y}, \varepsilon_{x_{0}s}\otimes \varepsilon_{\{y_{0}\}} \mid s \in S \}$,
$(\varepsilon_{X}\otimes \varepsilon_{Y\setminus  \{y_{0}\}}) \mathcal{T}(S \wr T) (\varepsilon_{X}\otimes \varepsilon_{Y\setminus  \{y_{0}\}})$
is generated by $\{ J_{X}\otimes \sigma_{c} \mid c \in \mathcal{R}_{Y\setminus \{y_0\}} \} \cup \{ \sigma_{s}\otimes I_{Y\setminus \{y_0\}} \mid s \in S\}$. Thus, we can determine a basis of $\varepsilon_{X}\otimes \varepsilon_{Y\setminus  \{y_{0}\}} \mathcal{T}(S \wr T) \varepsilon_{X}\otimes \varepsilon_{Y\setminus  \{y_{0}\}}$ with respect to the set of basic relations of $\mathcal{C}(\mathcal{T}(T))_{Y\setminus \{y_0\}}$.

%%%%%%%%%%%%%%%%%%%%%%%%%%%%%%%%%%%%%%%%%%%%%%%%%%%%%%%%%%%%%%%%%%%%%%%%%%%%%%%%%
\subsection{A basis of $\mathcal{A}(T_{y_{0}}) = \mathcal{T}(Y, T, y_0)$}\label{sec:basis}
By Theorem \ref{thm:coherent-terwilliger}, $\mathcal{A}(\mathcal{C}_{y_{0}}) = \mathcal{T}(Y, T, y_0)$.
In order to find a basis of $\mathcal{T}(Y, T, y_0)$, it is enough to know all basic relations of $\mathcal{C}_{y_{0}}$.
In particular, we focus on $\Delta \times \Gamma \in T_{y_0} \text{or} \not \in T_{y_0}$ for $\Delta,\Gamma \in \mathrm{Fib}(\mathcal{C}_{y_0})$ of size 2.

\begin{lem}\label{lem:222}
If $\mathcal{C}=(Y,T)$ belongs to case (1) or (3) in Theorem \ref{thm:3type},
then $\Delta \times \Gamma \not \in T_{y_0}$ for $\Delta,\Gamma \in \mathrm{Fib}(\mathcal{C}_{y_0})$ of size 2.
\end{lem}
\begin{proof}
In the case of Theorem \ref{thm:3type}(1), each $t \in T_2$ is represented by $t_1 u$ for some $t_1 \in T_1$.
For $t, t' \in T_2$,
$\sigma_t \sigma_{t'^*} = \sigma_{t_1} \sigma_{u}\sigma_{u^*} \sigma_{t_1'^*} = \sigma_{t_1} (2\sigma_{1_Y} + \sigma_{u}) \sigma_{t_1'^*} = 2\sigma_{t_1 t_1'} + \sigma_{t_1} \sigma_{u} \sigma_{t_1'^*}$
So the coefficient of $\sigma_{t_1} \sigma_{u} \sigma_{t_1'^*}$ implies that $\Delta \times \Gamma \not \in T_{y_0}$ for $\Delta,\Gamma \in \mathrm{Fib}(\mathcal{C}_{y_0})$ of size 2.

In the case of Theorem \ref{thm:3type}(3), clearly $\Delta \times \Gamma \not \in T_{y_0}$ for $\Delta,\Gamma \in \mathrm{Fib}(\mathcal{C}_{y_0})$ of size 2.
\end{proof}

\begin{lem}\label{lem:222333}
Suppose that $\mathcal{C}=(Y,T)$ belongs to case (2) in Theorem \ref{thm:3type}.
For distinct $i_1, i_2 \in I$, if fibers $\Delta \subseteq Y_{i_1}$ and $\Gamma \subseteq Y_{i_2}$, then $\Delta \times \Gamma \in T_{y_0}$.
\end{lem}
\begin{proof}
First, we show that $(C_{(H)})_{Y'} = (C_{y_0})_{Y'}$, where $Y' = \cup_{t \in T_2}y_0 t$.
Since $\mathcal{A}(\mathcal{C}_{y_0}) = \mathcal{T}(T) = \langle \{ \varepsilon_{y_0 t_1} \sigma_t \varepsilon_{y_0 t_2} \mid    t_1, t_2, t \in T \} \rangle$, $\varepsilon_{Y'} \mathcal{A}(\mathcal{C}_{y_0}) \varepsilon_{Y'} = \langle \{ \varepsilon_{y_0 t_1} \sigma_t \varepsilon_{y_0 t_2} \mid    t_1, t_2 \in T_2, t \in T \} \rangle$.
By thin residue extension,
$\mathcal{A}(\mathcal{C}_{(H)}) = \langle \{ \varepsilon_{\Delta} \sigma_t \varepsilon_{\Delta t} \mid \Delta \in Y/H, t \in T  \} \rangle$.
So we have $\varepsilon_{Y'} \mathcal{A}(\mathcal{C}_{(H)}) \varepsilon_{Y'} = \langle \{ \varepsilon_{\Delta} \sigma_t \varepsilon_{\Delta t} \mid \Delta \in Y/H, t \in T, \Delta= y_0 t', \Delta t= y_0 t''  ~\text{for some}~ t', t'' \in T_2 \} \rangle$. Thus, $(C_{(H)})_{Y'} = (C_{y_0})_{Y'}$.

Now we consider $(C_{y_0})_{Y'}$. Note that $Y' = \cup_{i \in I \setminus \{i_0\}} Y_i$ and $Y_{i_0} = \cup_{t \in T_1} y_0 t$, where $i_0$ is a class of $I$ such that $y_0 \in Y_{i_0}$.
Since $\mathcal{C}_{(H)} = \boxplus_{i \in I} \mathcal{C}_i$ and $(C_{(H)})_{Y'} = (C_{y_0})_{Y'}$, if fibers $\Delta \subseteq Y_{i_1}$ and $\Gamma \subseteq Y_{i_2}$ for distinct $i_1, i_2 \in I\setminus \{i_0\}$, then $\Delta \times \Gamma \in T_{y_0}$.
Clearly, for $i_0$ and $i_1 \in I \setminus \{i_0\}$, if fibers $\Delta \subseteq Y_{i_0}$ and $\Gamma \subseteq Y_{i_1}$, then $\Delta \times \Gamma , \Gamma \times \Delta \in T_{y_0}$.
\end{proof}

%%%%%%%%%%%%%%%%%%%%%%%%%%%%%%%%%%%%%%%%%%%%%%%%%%%%%%%%%%%%%%%%%%%%%%%%%%%%%%%%%
\subsection{Central primitive idempotents of $\mathcal{T}(X\times Y, S \wr T, (x_0, y_0))$}\label{sec:com}

Set $F^{(t)}=(x_{0},y_{0})\bar{t}=(X,y_{0}t)$ and
$U^{(t)}=(S\wr T)_{(x_{0},y_{0})\bar{t}}$ for $t\in T$.
If $t \in T_1$, then $(F^{(t)},U^{(t)})$ is isomorphic to $(X,S)$.
If $t \in T_2$, then $(F^{(t)},U^{(t)})$ is isomorphic to the wreath product of $(X,S)$ by the trivial scheme of degree $2$.

For $\chi \in \mathrm{Irr}(\mathcal{T}(U^{(1_Y)}))^\times$, define
$$\tilde{e}_{\chi} = e_{\chi}\otimes \varepsilon_{\{y_{0}\}}\in \mathcal{T}(S\wr T).$$

For $t \in T_1 \setminus\{1_Y\}$ and
$\varphi \in \mathrm{Irr}(\mathcal{A}(U^{(t)}))^\times$, define
$$\bar{e}_{\varphi} = e_{\varphi}\otimes \varepsilon_{y_{0}t} \in \mathcal{T}(S\wr T).$$

For $t \in T_{2}$ and
$\psi \in \mathrm{Irr}(\mathcal{A}(S))^\times$, define
$$\hat{e}_{\psi} = e_{\psi}\otimes \varepsilon_{y_{0}t} \in \mathcal{T}(S\wr T).$$

Then they are idempotents of $\mathcal{T}(S\wr T)$.

\begin{lem}\label{lem:32}(\cite{hkm}, Lemma 4.2 and 4.4)
For $\chi \in \mathrm{Irr}(\mathcal{T}(U^{(1_Y)}))^\times$,
$\tilde{e}_{\chi}$ is a central primitive idempotent of $\mathcal{T}(S \wr T)$.
\end{lem}

\begin{lem}\label{lem:33}(\cite{hkm}, Lemma 4.3 and 4.4)
For $t \in T_{1} \setminus \{1_Y\}$ and
$\varphi \in \mathrm{Irr}(\mathcal{A}(U^{(t)}))^\times$,
$\bar{e}_{\varphi}$ is a central primitive idempotent of $\mathcal{T}(S \wr T)$.
\end{lem}

By mimicking the proof of Lemma \ref{lem:33}, we get the following lemma.

\begin{lem}\label{lem:34}
For $t \in T_{2}$ and
$\psi \in \mathrm{Irr}(\mathcal{A}(S))^\times$,
$\hat{e}_{\psi}$ is a central primitive idempotent of $\mathcal{T}(S \wr T)$.
\end{lem}
\begin{proof}
First, we show that $\hat{e}_{\psi}$ commutes with $\sigma_{s}\otimes I_{Y}$, $J_{X}\otimes \sigma_{u}$ ($u \in
T\setminus\{1_Y\}$), $\varepsilon_{x_{0}s}\otimes \varepsilon_{\{y_{0}\}}$, and $I_{X}\otimes \varepsilon_{y_{0}u}$ ($u\in T\setminus\{1_Y\}$).
For $s \in S$, $\hat{e}_{\psi} (\sigma_{s}\otimes I_{Y}) = \sum_{u\in T} \hat{e}_{\psi}(\sigma_{s}\otimes \varepsilon_{y_0 u}) = (e_{\psi}\otimes \varepsilon_{y_0 t} )(\sigma_{s}\otimes \varepsilon_{y_0 t})$.
Since $e_{\psi}$ commutes with $\sigma_s$, we have $\hat{e}_{\psi} (\sigma_{s}\otimes I_{Y}) = (\sigma_{s}\otimes I_{Y}) \hat{e}_{\psi}$.
Since $t \neq 1_Y$, we have $\hat{e}_{\psi}(\varepsilon_{x_{0}s}\otimes \varepsilon_{\{y_{0}\}}) = (\varepsilon_{x_{0}s}\otimes \varepsilon_{\{y_{0}\}})\hat{e}_{\psi} = 0$.
Since $e_{1_{\mathcal{A}(S)}}=|X|^{-1}J_X$ and $e_{1_{\mathcal{A}(S)}}e_\psi=e_\psi e_{1_{\mathcal{A}(S)}}=0$, we have $\hat{e}_{\psi}(J_X\otimes \sigma_u) = (J_X\otimes \sigma_u)\hat{e}_{\psi} = 0$.
Also, $\hat{e}_{\psi}(I_{X}\otimes \varepsilon_{y_{0}u}) = (I_{X}\otimes \varepsilon_{y_{0}u})\hat{e}_{\psi}$ is trivial.

Now we show that $\hat{e}_{\psi}$ is primitive.
The map  $\pi:\mathcal{T}(S\wr T) \rightarrow \hat{e}_{\psi}\mathcal{T}(S\wr T)$ is a projection.
Actually, $\hat{e}_{\psi}\mathcal{T}(S\wr T)$ is naturally isomorphic to $e_{\psi}\mathcal{A}(S)$.
Since $e_{\psi}$ is a central primitive idempotent of $\mathcal{A}(S)$, $\hat{e}_{\psi}$ is primitive.
\end{proof}

From now on, we define the other central primitive idempotents of $\mathcal{T}(S \wr T)$.
Suppose that $(Y,T)$ belongs to case (1) or (3) in Theorem \ref{thm:3type}.
We define the following matrices $G_{y_0 t, y_0 t'}$ for $t, t' \in T_2$.
Let $G_{y_0 t, y_0 t'} = \frac{1}{2|X|}J_{X}\otimes (J_{\{y_{t(1)}\}, \{y_{t'(1)}\}} + J_{\{y_{t(2)}\}, \{y_{t'(2)}\}} - J_{\{y_{t(1)}\}, \{y_{t'(2)}\}} - J_{\{y_{t(2)}\}, \{y_{t'(1)}\}})$, where $y_0 t = \{ y_{t(1)}, y_{t(2)} \}$ and $y_0 t' = \{ y_{t'(1)}, y_{t'(2)} \}$.
It is easy to see that $\{G_{y_0 t, y_0 t'} \mid t, t' \in T_2 \}$ is a linearly independent subset of $\mathcal{T}(S \wr T)$.

\begin{lem}\label{lem:ideal}
$\langle\{G_{y_0 t, y_0 t'} \mid t, t' \in T_2 \}\rangle$ is an ideal. Moreover, $\langle\{G_{y_0 t, y_0 t'} \mid t, t' \in T_2 \}\rangle \cong \mathrm{Mat}_{T_2}(\mathbb{C})$.
\end{lem}
\begin{proof}
First, we prove that $\sigma_u G_{y_0 t, y_0 t'}, G_{y_0 t, y_0 t'} \sigma_u \in \langle\{G_{y_0 t, y_0 t'} \mid t, t' \in T_2 \}\rangle$ for $u \in S \wr T$.
Since $(\sigma_u)_{y_0 h, y_0t} \neq 0$ for some $h \in T_2$ and $y_0 h \times y_0t \not \in \mathcal{C}(\mathcal{T}(S \wr T))$ by Lemma \ref{lem:222}, we have $\sigma_u G_{y_0 t, y_0 t'} = \pm G_{y_0 h, y_0 t'}$.
Similarly, $G_{y_0 t, y_0 t'} \sigma_u  \in \langle\{G_{y_0 t, y_0 t'} \mid t, t' \in T_2 \}\rangle$ is proved.
For $u \in S \wr T$ and $t, t' \in T_2$, clearly $\varepsilon_{(x_0,y_0)u} G_{y_0 t, y_0 t'} = \delta_{(x_0,y_0)u (X\times y_0 t)}G_{y_0 t, y_0 t'} \in \langle\{G_{y_0 t, y_0 t'} \mid t, t' \in T_2 \}\rangle$
and $G_{y_0 t, y_0 t'}\varepsilon_{(x_0,y_0)u} \in \langle\{G_{y_0 t, y_0 t'} \mid t, t' \in T_2 \}\rangle$. So $\langle\{G_{y_0 t, y_0 t'} \mid t, t' \in T_2 \}\rangle$ is an ideal.

Now we prove that $G_{y_0 t, y_0 t'} G_{y_0 t''', y_0 t''} = \delta_{t't'''}G_{y_0 t, y_0 t''}$.
It is enough to show that $G_{y_0 t, y_0 t'}G_{y_0 t', y_0 t''} = G_{y_0 t, y_0 t''}$.
By calculation, we have $G_{y_0 t, y_0 t'}G_{y_0 t', y_0 t''} = \frac{1}{2|X|}J_{X}\otimes (J_{\{y_{t(1)}\}, \{y_{t'(1)}\}} + J_{\{y_{t(2)}\}, \{y_{t'(2)}\}} - J_{\{y_{t(1)}\}, \{y_{t'(2)}\}} - J_{\{y_{t(2)}\}, \{y_{t'(1)}\}}) \\
\frac{1}{2|X|}J_{X}\otimes (J_{\{y_{t'(1)}\}, \{y_{t''(1)}\}} + J_{\{y_{t'(2)}\}, \{y_{t''(2)}\}} - J_{\{y_{t'(1)}\}, \{y_{t''(2)}\}} - J_{\{y_{t'(2)}\}, \{y_{t''(1)}\}}) = G_{y_0 t, y_0 t''}$.

Finally, we prove that $\langle\{G_{y_0 t, y_0 t'} \mid t, t' \in T_2 \}\rangle \cong \mathrm{Mat}_{T_2}(\mathbb{C})$.
For $t, t' \in T_2$, let $e_{tt'}$ be the $|T_2| \times |T_2|$ matrix whose $(t, t')$-entry is $1$ and whose other entries are all zero.
Then the linear map $\varphi: \langle\{G_{y_0 t, y_0 t'} \mid t, t' \in T_2 \}\rangle \rightarrow \mathrm{Mat}_{T_2}(\mathbb{C})$ defined by
$\varphi(G_{y_0 t, y_0 t'}) = e_{tt'}$ is an isomorphism.
\end{proof}

Define $$e_\eta = \sum_{t \in T_2} \frac{1}{2|X|}J_{X}\otimes (\varepsilon_{y_{0}t} - \overline{\varepsilon_{y_{0}t}}),$$
where $\overline{\varepsilon_{y_{0}t}} := J_{\{y_{t(1)}\}, \{y_{t(2)}\}} + J_{\{y_{t(2)}\}, \{y_{t(1)}\}}$.
Then by Lemma \ref{lem:ideal}, $e_\eta$ is a central primitive idempotent of $\mathcal{T}(S\wr T)$.

\vskip5pt
Suppose that $(Y,T)$ belongs to case (2) in Theorem \ref{thm:3type}.
Put $y_0 \in Y_{i_0}$. For each $i \in I \setminus \{i_0\}$, we consider the set $U_i := \{ t \in T_2 \mid y_0 t \subseteq Y_i \}$.
Define the following matrices $G_{y_0 t, y_0 t'}$ for $t, t' \in U_i$.
Let $G_{y_0 t, y_0 t'} = \frac{1}{2|X|}J_{X}\otimes (J_{\{y_{t(1)}\}, \{y_{t'(1)}\}} + J_{\{y_{t(2)}\}, \{y_{t'(2)}\}} - J_{\{y_{t(1)}\}, \{y_{t'(2)}\}} - J_{\{y_{t(2)}\}, \{y_{t'(1)}\}})$.
According to process in the proof of Lemma \ref{lem:ideal}, we can prove that $\langle\{G_{y_0 t, y_0 t'} \mid t, t' \in U_i \}\rangle$ is an ideal.
For each $i \in I \setminus \{i_0\}$, define
$$e_{\eta_i} = \sum_{t \in U_i} \frac{1}{2|X|}J_{X}\otimes (\varepsilon_{y_{0}t} - \overline{\varepsilon_{y_{0}t}}),$$
where $\overline{\varepsilon_{y_{0}t}} := J_{\{y_{t(1)}\}, \{y_{t(2)}\}} + J_{\{y_{t(2)}\}, \{y_{t(1)}\}}$.
Then $e_{\eta_i}$ is a central primitive idempotent of $\mathcal{T}(S\wr T)$.
We denote $\sum_{i \in I\setminus \{i_0\}} e_{\eta_i}$ by $e_\eta$.

\begin{lem}\label{lem:35}
The sum of $e_{1_{\mathcal{T}(S\wr T)}}$, $\tilde{e}_{\chi}$'s, $\bar{e}_{\varphi}$'s, $\hat{e}_{\psi}$'s and $e_\eta$ is the identity element.
\end{lem}
\begin{proof}
It is easy to see that
\begin{eqnarray*}
  e_{1_{\mathcal{T}(S\wr T)}} &=& e_{1_{\mathcal{T}(U^{(1_Y)})}}\otimes \varepsilon_{\{y_{0}\}}
  + \sum_{t\in T_1 \setminus \{1_Y\}}\frac{1}{|X|}\varepsilon_{F^{(t)}}J_{X\times Y}\varepsilon_{F^{(t)}} + \sum_{t\in T_2}\frac{1}{2|X|}\varepsilon_{F^{(t)}}J_{X\times Y}\varepsilon_{F^{(t)}},
  \end{eqnarray*}

  \begin{eqnarray*}
  \sum_{\chi\in \mathrm{Irr}(\mathcal{T}(U^{(1_Y)}))^\times}\tilde{e}_\chi
  &=& \varepsilon_{F^{(1)}}I_{X\times Y}\varepsilon_{F^{(1)}}
  -e_{1_{\mathcal{T}(U^{(1_Y)})}}\otimes \varepsilon_{\{y_0\}},
  \end{eqnarray*}

  \begin{eqnarray*}
  \sum_{\varphi\in \mathrm{Irr}(\mathcal{A}(U^{(t)}))^\times}\bar{e}_\varphi
  &=& \varepsilon_{F^{(t)}}I_{X\times Y}\varepsilon_{F^{(t)}}
  -\frac{1}{|X|}\varepsilon_{F^{(t)}}J_{X\times Y}\varepsilon_{F^{(t)}}
\end{eqnarray*}
for each $t \in T_1 \setminus \{1_Y\}$,

\begin{eqnarray*}
  \sum_{\psi \in \mathrm{Irr}(\mathcal{A}(S))^\times}\hat{e}_\psi
  &=& \varepsilon_{F^{(t)}}I_{X\times Y}\varepsilon_{F^{(t)}}
  -\frac{1}{|X|}(J_{X\times y_{t(1)}}+J_{X\times y_{t(2)}})
\end{eqnarray*}
for each $t \in T_2$,

\begin{eqnarray*}
e_{\eta} = \sum_{t \in T_2} \frac{1}{2|X|}J_{X}\otimes (\varepsilon_{y_{0}t} - \overline{\varepsilon_{y_{0}t}}).
\end{eqnarray*}

Thus, we have
\begin{eqnarray*}
  e_{1_{\mathcal{T}(S\wr T)}}
  +\sum_{\chi\in \mathrm{Irr}(\mathcal{T}(U^{(1_Y)}))^\times}\tilde{e}_\chi
  +\sum_{t \in T_1\setminus\{1_Y\}} \sum_{\varphi\in \mathrm{Irr}(\mathcal{A}(U^{(t)}))^\times}\bar{e}_\varphi + \\
  \sum_{t \in T_2} \sum_{\psi\in \mathrm{Irr}(\mathcal{A}(S))^\times}\hat{e}_\psi
  + e_{\eta} = I_{X\times Y}.
\end{eqnarray*}
\end{proof}

\section{Main result}\label{sec:maintheorem}
In conclusion, we have determined the set of all central primitive idempotents of Terwilliger algebras of wreath products by quasi-thin schemes.
Combining Section \ref{sec:main} and Theorem 4.1 of \cite{hkm} gives the following theorem.

\begin{thm}\label{thm:quasithinthm}
Let $(X,S)$ and $(Y,T)$ be association schemes. Suppose that $(Y,T)$   is a quasi-thin scheme or a one-class scheme. Fix $x_{0}\in X$ and
  $y_{0}\in Y$, and consider the wreath product $(X\times Y, {S\wr T})$. Then
\begin{enumerate}
\item[(1)] If $(Y,T)$ is a thin scheme or a one-class scheme, then
\begin{eqnarray*}
\{e_{1_{\mathcal{T}(S\wr T)}}\} &\cup& \{\tilde{e}_{\chi}
\mid \chi\in \mathrm{Irr}(\mathcal{T}(U^{(1_Y)}))^\times\}\\
&\cup& \bigcup_{t \in T \setminus \{1_Y\}} \{ \bar{e}_{\varphi} \mid
\varphi \in \mathrm{Irr}(\mathcal{A}(U^{(t)}))^\times \}
\end{eqnarray*}
is the set of all central primitive idempotents of $\mathcal{T}(X\times Y, \ {S\wr T}, \ (x_{0},y_{0}))$.
\item[(2)] If $(Y,T)$ has $T^{\bot} \subseteq T_2$ or $|T^{\bot}| \geq 2$, then
\begin{eqnarray*}
\{e_{1_{\mathcal{T}(S\wr T)}}\} &\cup& \{\tilde{e}_{\chi}
\mid \chi\in \mathrm{Irr}(\mathcal{T}(U^{(1_Y)}))^\times \}\\
&\cup& \bigcup_{t \in T_1 \setminus \{1_Y\}} \{ \bar{e}_{\varphi} \mid
\varphi \in \mathrm{Irr}(\mathcal{A}(U^{(t)}))^\times \}\\
&\cup& \bigcup_{t \in T_2} \{ \hat{e}_{\psi} \mid
\psi \in \mathrm{Irr}(\mathcal{A}(S))^\times \} \cup \{e_\eta \}
\end{eqnarray*}
is the set of all central primitive idempotents of $\mathcal{T}(X\times Y, \ {S\wr T}, \ (x_{0},y_{0}))$.
\item[(3)] If $(Y,T)$ has $|T^{\bot}| = 1$ and $T^{\bot} \subseteq T_1$, then
\begin{eqnarray*}
\{e_{1_{\mathcal{T}(S\wr T)}}\} &\cup& \{\tilde{e}_{\chi}
\mid \chi\in \mathrm{Irr}(\mathcal{T}(U^{(1_Y)}))^\times \}\\
&\cup& \bigcup_{t \in T_1 \setminus \{1_Y\}} \{ \bar{e}_{\varphi} \mid
\varphi \in \mathrm{Irr}(\mathcal{A}(U^{(t)}))^\times \}\\
&\cup& \bigcup_{t \in T_2} \{ \hat{e}_{\psi} \mid
\psi \in\mathrm{ Irr}(\mathcal{A}(S))^\times \} \\
&\cup& \{e_{\eta_i} \mid i \in I \setminus \{i_0\}\}
\end{eqnarray*}
is the set of all central primitive idempotents of $\mathcal{T}(X\times Y, \ {S\wr T}, \ (x_{0},y_{0}))$.
\end{enumerate}
\end{thm}

\thanks{\textbf{Acknowledgements} The author would like to thank an anonymous referee for valuable comments.}

\bibstyle{plain}

\end{document}